\documentclass[12pt]{article}
\usepackage[utf8]{inputenc}
\usepackage{latexsym,amsfonts,amssymb,amsthm,amsmath,extarrows}
\usepackage{mathrsfs}
\usepackage{cite}
\usepackage{graphicx}
\usepackage{natbib}
\usepackage{enumerate}
\usepackage{bm}

\newtheorem{theorem}{Theorem}
\newtheorem{lemma}{Lemma}
\newtheorem{corollary}{Corollary}
\newtheorem{proposition}{Proposition}

\usepackage[noend]{algpseudocode}
\usepackage{algorithmicx,algorithm}

\newcommand{\pr}{\mathrm{P}}
\newcommand{\E}{\mathrm{E}}

\usepackage{geometry}
\RequirePackage[colorlinks,citecolor=blue,urlcolor=blue]{hyperref}

\title{Concentration inequalities of MLE and robust MLE}
\author{Xiaowei Yang$^{1}$~~Xinqiao Liu$^2$~~Haoyu Wei$^{3}$\thanks{Corresponding authors: Haoyu Wei (h8wei@ucsd.edu). The Yang X. (yxw@stu.scu.edu.cn) and Liu X. (xinqiao.liu@tju.edu.cn) are co-first authors.}}
\date{\small 1. College of Mathematics, Sichuan University, Chengdu, China;\\
2. School of Education, Tianjin University, Tianjin, China;\\
3. Department of Economics, University of California San Diego, La Jolla, USA.}

\begin{document}

\maketitle

\begin{abstract}
The Maximum Likelihood Estimator (MLE) serves an important role in statistics and machine learning. In this article, for i.i.d. variables, we obtain constant-specified and sharp concentration inequalities and oracle inequalities for the MLE only under exponential moment conditions. Furthermore, in a robust setting, the sub-Gaussian type oracle inequalities of the log-truncated maximum likelihood estimator are derived under the second-moment condition.\\


\end{abstract}

\section{Introduction}


%

Let $\left\{X_{i}\right\}_{i=1}^n$ be i.i.d. variables in a metric space $(\mathcal{X}, d)$ with the population distribution $X \sim \mathbb{P}_{\theta^*},(\theta^* \in \mathbb{R})$. Given a density $p({x},\theta )$ and $l(x,\theta ): = -\log p({x},\theta )$, the maximum likelihood estimator (MLE) is a function of the data
\begin{equation}\label{eq:MLE}
    \widehat{\theta} : = \mathop {\arg\min }\limits_{\theta  \in {\Theta}} \left[\frac{1}{n}\sum\limits_{i = 1}^n {l(X_{i}, \theta )} \right]=f(X_{1},\cdots,X_{n}).
\end{equation}
The primitive and classical results of MLE are in an asymptotic view \citep{cramer1946mathematical}, i.e. obtaining the limit law of $\widehat{\theta}$ and then doing inference for the true parameter ${\theta ^*}  =  {\arg \min }_{\theta  \in \Theta } {\rm{E}}[l(X,\theta )]$ based on the asymptotic confidence intervals. Lately, the asymptotic normality of MLE is presented and taught in most textbooks of mathematical statistics; see \cite{shao2010mathematical} and references therein. Recently, machine learners care more about finite sample theories for $\widehat{\theta}$ derivation from its true value ${\theta ^*}$.
Given the metric $d$, it is of interest to find non-asymptotical and constant-shaper of the error bound $|\widehat{\theta}-{\theta ^*}|$ for all $n$ with the $1/{\sqrt n}$ convergence rate of $\widehat{\theta}$. \cite{Miao2020} only obtained the sub-Gaussian concentration of $|\widehat{\theta}-{\rm{E}}\widehat{\theta}|$ under some stronger conditions. But he did not consider ${\rm{E}}\widehat{\theta}-{\theta ^*}$. Note that $\E \widehat{\theta}$ may not always equal to $\theta^*$, it is better to consider $\widehat{\theta} - {\theta ^*}=(\widehat{\theta}-{\rm{E}}\widehat{\theta})+({\rm{E}}\widehat{\theta}-{\theta ^*})$ instead of $\widehat{\theta}-{\rm{E}}\widehat{\theta}$. Moreover, rather than using the sub-Gaussian data, we will just 
require the data are the sub-exponential or with finite moment conditions, which is larger than sub-Gaussian families and only requires the existence of moment generating function (MGF) at a neighbourhood of zero. What we will do, to fill the existent gap, is dealing with sharp and constant determined concentration inequalities for MLE without or with truncation for $n<\infty$.

In the last 50 years, classical robust statistics is a traditional topic that has been well studied, starting from  Huber's pioneering work in the asymptotical view; see \cite{huber1964robust}. In \cite{lerasle2019lecture}, an estimator is called robust if it behaves nicely when the data may be corrupted with measurement error. These days, robust learning is revitalized and invigorating, and the new estimators have sub-Gaussian behavior in a relaxed setting where the Gaussian assumption (exponential moment) is replaced by some finite moment hypotheses; see \cite{lerasle2019lecture,sun2020adaptive} and references therein. Robust M-estimators for the mean of corrupted observation data, are extensively popular in recent robust statistical learning; see \cite{Catoni2012,sun2021we,zhang2022sharper,yao2022asy}. The machine learners aim to propose robust estimators that enjoy finite sample theory. The sub-Gaussian data assumption is not an impressive extension from bounded data to unbounded data for this paper's purpose, since all $k$th-moments of sub-Gaussian data exist. There are data distributions without higher moments in finance statistics and extreme statistics. For example, the Italian economist Vilfredo Pareto studied the distribution of wealth in England in 1882, that the top 20$\%$ of the population owned 80$\%$ of society's wealth, and this phenomenon can be described by a Pareto distribution with only some finite moments. The Pareto distribution is a skewed and fat-tailed distribution that is widely used in economic research; see \cite{arnold2014pareto,Weinberg2017}.

The existing literature is only for sub-Gaussian conditions, and the constants of concentration inequalities are unknown.
We improve it to two exponential conditions in Theorem \ref{thm:MLE}. We also propose a log-truncated MLE, which does not require exponential moment conditions, but only second moment condition implies sub-Gaussian concentration in Theorem \ref{eq:cantoniG}. The specific contributions:
\begin{itemize}

\item Through a different proof, our first result contributes a constant-specified concentration inequality for MLE with sub-Gaussian condition of data, while the constant in Theorem 3.3 of \cite{Miao2020} is unknown. The proof is by a sharp concentration inequality (Corollary \ref{cor:fxcon}).

\item By Corollary \ref{lem:maurer2021some}, our second contribution improves \cite{Miao2020} if the data have finite sub-exponential moment.  \cite{Miao2020} did not give the concentration of MLE in the sub-exponential case, and he also did not consider the bias bound of the MLE and error bound of $|\widehat{\theta}-{\theta ^*}|$.

\item In robust setting, to avoid making the exponential moment conditions, we choose the log-truncated MLE to derive the sub-Gaussian type high probability error bounds for the robust estimator under a broad class of log-likelihood function.


\end{itemize}

\textbf{Notations}. For $\theta \in \mathbb{R}^p$, ${\| \theta  \|_{\ell_q}}: ={( {\sum_{j = 1}^p {\theta_j^q} } )^{1/q}}$. A centered random variable (r.v.) $X$ is called sub-Gaussian if ${\rm{E}}{e^{s X}} \le {e^{{s^2}{{\sigma}^2}/2}}$ for $\forall~s \in \mathbb{R}$, where the quantity ${\sigma}>0$ is named as the sub-Gaussian parameter. And we denote it as $X \sim \operatorname{subG}({\sigma^2})$. Similarly, a r.v. $X$ with $\E X = 0$ is sub-exponential with two positive parameters $(\lambda, \alpha)$, if its MGF satisfies ${\mathrm{{E}}}e^{s X} \leq e^{\frac{s^{2}\lambda ^{2}}{2}}~\text { for all }|s|<{1}/{\alpha}$. We denote as $X \sim \operatorname{subE}(\lambda, \alpha)$. Define the $L_p$-norm of r.v. $X$ as ${\left\| X \right\|_p} = {({\rm{E}}|X{|^p})^{1/p}}$. And sub-Gaussian and sub-exponential norms $\|\cdot\|_{\theta_{1}}$ and $\|\cdot\|_{\theta_{2}} $ for zero-mean r.v. $X$ are defined as
$\|X\|_{\theta_{1}}=\sup _{p \geq 1} \frac{\|X\|_{p}}{{(p!)}^{1/p}}$ and $\|X\|_{\theta_{2}}=\sup _{p \geq 1} [\frac{{\rm{E}}X^{2p}}{(2p-1)!!}]^{1/{(2p)}}$ (pages 6 and 23 in \cite{buldygin2000metric}). Assume that the negative log-likelihood function $l(x,\theta)$ is third continuously differentiable in $\theta$ for $\forall x \in \mathcal{X}$. For brevity, we use $\dot{l}(x,\theta)$, $\ddot{l}(x,\theta)$, and $\dddot{l}(x,\theta)$ to denote the first, second, and third partial derivatives of $l(x,\theta)$ with respect to (w.r.t.) $\theta$  respectively. Fisher information is defined as $I(\theta) = {\rm{E}}\ddot{l}(X,\theta)$.

\section{From MLE to Log-truncated MLE}

\subsection{Non-asymptotic results of MLE}

\cite{Miao2020} adopted the classical logarithmic Soblev inequalities method to bound $|\widehat{\theta}-{\rm{E}}\widehat{\theta}|$ by checking the bounded difference condition of $\widehat{\theta} :=f(X_{1},\cdots,X_{n})$ from the following assumptions of loss function with Lipschitz constant ${c_l}>0$ w.r.t the metric $d(\cdot,\cdot)$ on $\mathcal{X}$:
 \begin{equation}\label{eq:condition}
    \text{(C.1).}~\mathop {\inf }\limits_{\theta  \in \Theta } \ddot{l}(x,\theta) \geq {c_H}>0~;\text{(C.2).}~|\dot{l}(x,\theta)-\dot{l}(y,\theta)| \le {c_l}d(x,y),~\forall \theta \in \Theta,~\forall x,y \in \mathcal{X}.
\end{equation}
The following inequality \eqref{eq5} is given in Lemma 3.2 of \cite{Miao2020}
\begin{align}\label{eq5}
    | \widetilde \theta - \widehat \theta | \leq  \frac{c_{l}d(X_{k},Y_{k})}{n c_{H}},
\end{align}
where $\widetilde{\theta}:=f(X_{1},\cdots,X_{k-1},Y_k,X_{k+1},\cdots,X_{n})$ is the MLE of $\theta^{*}$ in the case of data with perturbations of the $k$-th data point $Y_k$.
\eqref{eq5} satisfies the bounded difference conditions if $d(x,y)$ is bounded implies that the data is bounded, and McDiarmid's inequality is applied to get the sub-Gaussian concentration of $\widehat{\theta} =f(X_{1},\cdots,X_{n})$. However, exponential, Pareto and Gaussian distributions have unbounded support, and ${d({X_k},{Y_k})}$ is unbounded. Applying a new and shaper McDiarmid type inequalities with unbounded difference conditions for summations of independent sub-Gaussian or sub-exponential r.v., we show below that the MLE achieves the rate $1/{\sqrt n}$ constant-specified error bound of $|\widehat{\theta}-{\rm{E}}\widehat{\theta}|$ under mild conditions if data have sub-Gaussian or sub-exponential tails.

Let $Z=\left(Z_{1}, \ldots, Z_{n}\right)$ be a vector of independent r.v.s in a space $\mathcal{Z},$ and define $Z^{\prime}$ as an independent copy of $Z$. For any function $f: \mathcal{Z}^{n} \rightarrow \mathbb{R}$, it is of interest to study the concentration for $f(Z)$ about its expectation. For $w \in \mathcal{Z}$ and $k \in\{1, \ldots, n\}$, define the substitution operator $S_{w}^{k}: \mathcal{Z}^{n} \rightarrow \mathcal{Z}^{n}$ by
$S_{w}^{k} z := \left(z_{1}, \ldots, z_{k-1}, w, z_{k+1}, \ldots, z_{n}\right)$ and the centered conditional version of $f$ 
\begin{align}\label{eq:identy}
    D_{f, Z_k}(z) & :=  f\left(z_{1}, \ldots, z_{k-1}, Z_{k}, z_{k+1}, \ldots, z_{n}\right)-{\rm{E}} f\left(z_{1}, \ldots, z_{k-1}, Z_{k}^{\prime}, z_{k+1}, \ldots, z_{n}\right) \nonumber\\
    &=f\left(S_{Z_{k}}^{k} z\right)-{\rm{E}}f(S_{Z_{k}^{\prime}}^{k} z) = {\rm{E}}[ f\left(S_{Z_{k}}^{k} z\right)-f(S_{Z_{k}^{\prime}}^{k} z) \mid Z_k ].
\end{align}
The $D_{f, Z_{k}}(z)$ can be viewed as random-variable-valued functions $z \in \mathcal{X}^{n} \mapsto D_{f, Z_{k}}(z)$. \cite{maurer2021some} showed the following tail inequalities, which are applicable to derive concentration inequality for suprema of unbounded empirical processes.

\begin{lemma}[Theorems 3 and 4 in \cite{maurer2021some}]\label{lem1} Let $t>0$. {\rm{(a)}}. If $\{D_{f, X_k}(x)\}_{k=1}^n$ have finite $\|\cdot\|_{\psi_{2}}$-norm,
$\operatorname{P}\left\{f(X)-\mathrm{E}f(X)>t\right\} \leq e^{\frac{-t^{2}}{32e\sup_{x \in \mathcal{X}^{n}}\sum_{k}\| D_{f, X_k}(x)\|_{\psi_{2}}^{2}}}$; {\rm{(b)}}. If $\{D_{f, X_k}(x)\}_{k=1}^n$ have finite $\|\cdot\|_{\psi_{1}}$-norm,
$\operatorname{P}\left\{f(X)-\mathrm{E}f\left(X\right)>t\right\} \leq e^{\frac{-t^{2}}{4 e^{2}\sup_{x \in \mathcal{X}^{n}}\sum_k\| D_{f, X_k}(x)\|_{\psi_{1}}^{2}+2 e \max_{1 \le k \le n}\sup_{x \in \mathcal{X}^{n}}\| D_{f, X_k}(x)\|_{\psi_{1}} t}}$.
\end{lemma}

For example, if $f(x)=\sum_{i=1}^{n} x_{i}$, then $D_{f, X_k}(x)= X_{k}-\mathrm{E}X_{k}$ is independent of $x.$  We have $\sup_{x \in \mathcal{X}^{n}}\sum_{k=1}^{n}\| D_{f, X_k}(x)\|_{\psi_{2}}^{2}=\sum_{k=1}^{n}\|  X_{k}-\mathrm{E}\left[X_{k}\right]\|_{\psi_{2}}^{2}$. As a generalization of McDiarmid's inequality, Lemma \ref{lem1} does not require bounded assumption on r.v.s and $f(X)$ may be dependent sum of r.v.s.

Recently, Corollary 4 in \cite{lei2020non} obtained a constant-sharper sub-Gaussian concentration for $f(Z)$ by using the $\| \cdot \|_{{\theta_2}}$-norm of $\{D_{f,  Z_i}(z)\}_{i=1}^n$. See the following proposition.

\begin{proposition}\label{cor:fxcon}
If $\{D_{f,  Z_i}(z)\}_{i=1}^n$ have finite $\| \cdot \|_{{\theta_2}}$-norm for ${z \in \mathcal{Z}}$, we have $f(Z)- {\rm{E}}f(Z) \sim \operatorname{subG}({\rm{8}}\sup_{z \in \mathcal{Z}} \sum_{i = 1}^n {\left\| {{D_{ f,{Z_i}}}(z)} \right\|_{{\theta _2}}^2} )$ and
${\mathop{\rm P}\nolimits} \left\{ {f(Z) - {\rm{E}}f(Z) > t} \right\} \le e^{{{{ - {t^2}}}/{({{\rm{16}}\sup_{z \in \mathcal{Z}} \sum_{i = 1}^n {\| {{D_{f,{Z_i}}}(z)} \|_{{\theta _2}}^2} })}}}$, $t \ge 0$.
\end{proposition}

Based on $\| \cdot \|_{{\theta_1}}$-norm of $\{D_{f,  Z_i}(z)\}_{i=1}^n$, we present the constant-sharper sub-exponential concentration for $f(Z)$ via sub-Gamma conditions. A centralized r.v. $X$ is sub-Gamma ($X \sim \textrm{sub}\Gamma(\eta,M)$) with the variance factor $\eta>0$ and the scale parameter $M>0$  if
$\log ({\rm{E}}e^{sX})\leq \frac{\eta s^{2}}{2({1}-M|s|)},~\forall~0<|s|<M^{-1}$.

\begin{corollary}\label{lem:maurer2021some}
If $\{D_{f, Z_{i}}(z)\}_{i=1}^n$ have finite sub-exponential norm for all ${z \in \mathcal{Z}}$, then
$f(Z) - {\rm{E}}f(Z) \sim \operatorname{sub\Gamma} \left(2{\sup_{z \in \mathcal{Z}} \sum_{i = 1}^n {\left\| {{D_{f,{Z_i}}}(z)} \right\|_{{\theta_1}}^2} }, {\max_{i \in [n]}\sup_{z \in \mathcal{Z}} \left\| {{D_{f,{Z_i}}}(z)} \right\|_{{\theta_1}}}\right)$. Moreover,
\begin{align}\label{eq:SUBG}
    \operatorname{P} \left\{f(Z) - {\rm{E}}f(Z)>2\bigg( { \sup\limits_{z \in \mathcal{Z}} \sum_{i = 1}^n {\left\| {{D_{f,{Z_i}}}(z)} \right\|_{{\theta_1}}^2} }t\bigg)^{1/2}+{\max_{i \in [n]} \sup\limits_{z \in \mathcal{Z}} \left\| {{D_{f,{Z_i}}}(z)} \right\|_{{\theta_1}}} t\right\} \leq  e^{-t},t \ge 0,~\text{and}
\end{align}
\begin{align}\label{eq:SUBG1}
\operatorname{P} \left\{ {f(Z) - {\rm{E}}f(Z) > t} \right\} \le e^{{{{ - {t^2}}}/({{4\sup_{z \in {\cal Z}} \sum_{i = 1}^n {\| {{D_{f,{Z_i}}}(z)} \|_{{\theta_1}}^2}  + 2\max_{i \in [n]} \sup_{z \in {\cal Z}} {{\| {{D_{f,{Z_i}}}(z)} \|}_{{\theta_1}}}t}})}}.
\end{align}
\end{corollary}

\textbf{Remark 1.} The constant coefficients in the conclusion of Lemma \ref{lem1} are not good enough. For example, the constants in Theorem 3 in \cite{maurer2021some} for Gaussian distribution (i.e. r.v. $X\sim N(\mu,\sigma^{2})$ with mean $\mu$ and variance $\sigma^{2}$) and in Theorem 3 in \cite{maurer2021some} for Laplace distribution (i.e. r.v. $X\sim {\rm{Laplace}}(\mu,\lambda)$ with mean $\mu$ and $|X-\mu|\sim {\rm{Exp}}(\lambda^{-1})$ ) are
\begin{center}
    $\frac{64en\sigma^{2}}{\sqrt{\pi}}\sup_{p\geq1}\frac{\left[\Gamma(\frac{p+1}{2})\right]^{2/p}}{p}
    =\frac{64en\sigma^{2}}{\sqrt{\pi}}$, $4e^{2}n\lambda^{2}\sup_{p\geq1}\frac{(p!)^{2/p}}{p^{2}}+2e\lambda\sup_{p\geq1}\frac{(p!)^{1/p}}{p}
=4e^{2}n\lambda^{2}+2e\lambda$
\end{center}
respectively.  However, our Corollary \ref{cor:fxcon} improves the constants as  $16n\sigma^{2}$ in Gaussian distribution, and our Corollary \ref{lem:maurer2021some} improves the constants as $4n\lambda^{2}+2\lambda$ in Laplace distribution respectively.

The next theorem is our main result. We get finite sample bounds of MLE under the Hessian and Lipschitz conditions, derived by the proposed sharper and constant-specified McDiarmid types inequalities (Proposition \ref{cor:fxcon} and Corollary \ref{lem:maurer2021some}) with unbounded difference conditions.
\begin{theorem}[Sharp concentration inequalities of MLE]\label{thm:MLE}
     Suppose the MLE $\widehat \theta$ is a function of data in \eqref{eq:MLE}. Assume conditions in \eqref{eq:condition} for $l(x,\theta )$, then we get $\widehat \theta-\mathrm{E}\widehat \theta \sim \operatorname{subG}({16c_{l}^{2}/(c_{H}^{2}n^{2}) \sum_{k = 1}^n {\left\| d({X_k},{Y_k}) \right\|_{{\theta _2}}^2} })$
\begin{align}\label{eq:subG}
~\text{and    }~{\mathop{\rm P}\nolimits} \{|\widehat \theta-\mathrm{E}\widehat \theta|>t\} \le 2e^{\frac{ - {t^2}}{{16c_{l}^{2}/(c_{H}^{2}n^{2}) \sum_{k = 1}^n {\left\| d({X_k},{Y_k}) \right\|_{{\theta _2}}^2} }}},\quad \forall t \ge 0~\text{or}~
\end{align}
$\widehat \theta-\mathrm{E}\widehat \theta\sim \operatorname{sub\Gamma}\left({2 c_l^2}/{(c_H^2n^2)} \sum_{k = 1}^n {\left\| d({X_k},{Y_k}) \right\|_{{\theta _1}}^2}, {c_l}/{(c_Hn)}\max_{1\le k \le n}{\left\| d({X_k},{Y_k}) \right\|_{{\theta _1}}}\right)$ and
\begin{align}\label{eq:subE}
    \operatorname{P}\{|\widehat \theta-\mathrm{E}\widehat \theta|>t\} & \le 2e^{\frac{-t^{2}}{{4 c_l^2}/{(c_H^2n^2)} \sum_{k = 1}^n {\left\| d({X_k},{Y_k}) \right\|_{{\theta _1}}^2}+{2c_l}/{(c_Hn)}\max_{1\le k \le n}{\left\| d({X_k},{Y_k}) \right\|_{{\theta _1}}} t}}
\end{align}
granted that the sub-Gaussian or sub-exponential norms exist in the concentration inequality.
\end{theorem}
Theorem \ref{thm:MLE} improves and extends Theorem 3.3 in \cite{Miao2020} from unknown constants to known constants in exponential inequalities, and from sub-Gaussian to sub-exponential samples. Using the bias formula of MLE in \cite{mardia1999bias}, we obtain a high-probability error bound of $|\widehat{\theta}-{\theta ^*}|$.
\begin{corollary}[Sharp oracle inequalities of MLE]\label{cor:error}
Under conditions in Theorem \ref{thm:MLE} and $\Theta=\mathbb{R}^1$, with probability at least $1-\delta$ we have  for sub-Gaussian and sub-exponential samples, respectively,
\begin{small}
\begin{align*}
|\widehat{\theta}-{\theta ^*}|\le \frac{4c_{l}}{\sqrt{n}c_{H}}\sqrt{\frac{1}{n}\sum_{k=1}^{n}\|d(X_{k},Y_{k})\|_{\theta_{2}}^{2}\log(\frac{2}{\delta})}
+\frac{|\kappa(X,\theta^*)|}{2nI^{2}(\theta^*)} + o\left(\frac{1}{n}\right),
\end{align*}
\begin{align}\label{eq:subEbias}
|\widehat{\theta}-{\theta ^*}| \leq \frac{2c_{l}}{\sqrt{n}c_{H}}\sqrt{\frac{1}{n}\sum_{i=1}^{n}\|d(X_{k},Y_{k})\|_{\theta_{1}}^{2}\log(\frac{1}{\delta})}
+\frac{1}{n}\left(\frac{\max\limits_{1\leq k\leq n}\|d(X_{k},Y_{k})\|_{\theta_{1}}\log({\delta}^{-1})}{c_{H}/c_{l}}
+\frac{|\kappa(X,\theta^*)|}{2I^{2}(\theta^*)}\right) + o\left(\frac{1}{n}\right),
\end{align}
\end{small}
where $\kappa(X,\theta^*):=2{\rm{E}}[\dot{l}(X,\theta^*)\ddot{l}(X,\theta^*)]-{\rm{E}}\dddot{l}(X,\theta^*)$.
\end{corollary}

\subsection{Log-truncated MLE}

However, if the sub-Gaussian norm of data in Theorem \ref{thm:MLE} does not exist, then the sub-Gaussian behaviours of the MLE $\widehat \theta$ in Theorem \ref{thm:MLE} and Corollary \ref{cor:error} fail to hold.

Below we introduce a nice method proposed by \cite{Catoni2012}, which requires only a finite second moment through a log-truncation technique to obtain the sub-Gaussian behaviour of some robust MLE. It is completely different from the MLE and emprical mean estimator. It is well known that the empirical mean $\overline{\mu}_{n}$ from likelihood function of Gaussian data is just the solution $\theta\in\mathbb{R}$ of the equation $\sum_{i=1}^{n}(X_{i}-\theta)=0$.  Catoni's approach is to replace the left side of the previous equation with another strictly increasing function $Q_{n,\beta}(\theta)=\sum_{i=1}^{n}\psi(\beta\dot{l}(X_{i},\theta))$ on $\theta$, where $\beta\in\mathbb{R}$ is a tuning parameter and $\psi:\mathbb{R}\rightarrow\mathbb{R}$ is an antisymmetric increasing function. The idea is that if $X_{i}$ grows much faster than $\psi(x)$, the influence of outliers due to heavy tails will be weakened. There are many influence functions $\psi(x)$. A popular choice of $\psi(x)$ is
\begin{align}\label{eq1}
\psi(x)=\textrm{sign}(x)\log(1+|x|+|x|^{2}/2).
\end{align}
The intuition is that if $\psi(x)$ grows much slower than $x$ like some log function, then it reduces the value of the exponential-scaled outliers as the normal data points. The log-truncated function $\psi(x)$ remains to be unbounded, thus it largely retains the data fluctuation in an unbounded way, while the classical bounded truncated M-functions such as Huber loss function would miss a lot of information in the data. Based on \eqref{eq1}, let
${{\widehat Z}_\beta }(\theta ) = \frac{1}{{n\beta}}\sum_{i = 1}^n \psi(\beta\dot{l}(X_{i},\theta))$ for any $\theta \in \Theta \subset \mathbb{R}$. We aim to estimate ${\theta ^*}  =  {\arg \min }_{\theta  \in \Theta } {\rm{E}}[l(X,\theta )]$. Define the $Z$-estimator ${{\widehat \theta }_\beta }$ as the solution of the estimating equation
\begin{align}\label{eq:CANTONZ}
    {{\widehat \theta}_{\beta }} = \{ \theta :{{\widehat Z}_\beta }(\theta ) = 0,\theta  \in \mathbb{R}\} ,
\end{align}
where $\beta$ is the tuning parameter, which will be determined in the following theorem. Next, we establish the confidence interval and the convergence rate of ${{\widehat \theta }_{\beta } }$ for a sufficiently large sample.

\begin{theorem}\label{eq:cantoniG}
For $\widehat \theta_{\beta}$ defined as \eqref{eq:CANTONZ} from i.i.d. data $\{X_i\}_{i=1}^n$ with $l(x,\theta )$ in \eqref{eq:MLE}, and $c(x)$ is from
\begin{align}\label{equ9}
    |\dot{l}(x,\theta)-\dot{l}(x,\theta^*)| \leq c(x)|\theta-\theta^*|,~ c(x)>0.
\end{align}
Define $\theta _{+}$ and $\theta _{-}$ as the smallest and largest solutions of $n\beta^{2}{\rm{E}}c^{2}(X_{1})(\theta-\theta^*)^{2}+n\beta {\rm{E}}c(X_{1})|\theta-\theta^*|+n\beta^{2}I(\theta^*)+\log(\delta^{-1})=0$ under $I(\theta^*)<\infty$. Then, {\rm{(i)}}. When $n\geq\frac{4{\rm{E}}c^{2}(X_{1})\log(\delta^{-1})}{[{\rm{E}}c(X_{1})]^{2}-4\beta^{2}{\rm{E}}c^{2}(X_{1})I(\theta^*)}>0$,
\begin{center}
    $\mathrm{P} (\theta_{-} \leq {\widehat \theta }_{\beta } \leq \theta_{+}) \ge 1 - 2\delta \quad$ for any $\quad \delta \in(0,1/2).$
\end{center}
{\rm{(ii)}}. If $c(x)$ is related to $x$ and $\beta=\sqrt{\frac{I^{-1}(\theta^*)\log(\delta^{-1})}{n(1+\frac{4{\rm{E}}c^{2}(X_{1})
\log(\delta^{-1})}{n[{\rm{E}}c(X_{1})]^{2}-4{\rm{E}}c^{2}(X_{1})\log(\delta^{-1})})}}$, then with probability at least $1-2\delta$,
\begin{align}\label{eq:T1}
    |{{\widehat \theta }_\beta }-{\theta ^*}|<2\sqrt{\frac{I(\theta^*)\log(\delta^{-1})}{n[{\rm{E}}c(X_{1})]^{2}-4{\rm{E}}c^{2}(X_{1})
\log(\delta^{-1})}}.
\end{align}
{\rm{(iii)}}. If $c(x)\equiv c>0$, pick $\beta=\sqrt{\frac{2
I^{-1}(\theta^*)\log(\delta^{-1})}{n\left(1+({2\log(\delta^{-1})})/({n-2\log(\delta^{-1})})\right)}}$, then with probability at least $1-2\delta$
\begin{align}\label{eq:T1}
    |{{\widehat \theta }_\beta }-{\theta ^*}|<\sqrt{\frac{2I(\theta^*)\log(\delta^{-1})}{nc^{2}-2c^{2}
\log(\delta^{-1})}}.
\end{align}
\end{theorem}

\textbf{Remark 2.} Regarding $c(x)$ in Theorem \ref{eq:cantoniG}, we give examples of calculating $c(x)$ by three distributions: Gaussian, Pareto and Weibull distributions if $\Theta$ is a compact set excluded $\{0\}$.

(a) If $X_{1}\sim \text{N}(0,\sigma^{2})$, then
$l(x,\theta)=l(x,\sigma^{2})=-\frac{1}{2}\log(2\pi\sigma^{2})+\frac{x^{2}}{2\sigma^{2}},
~\dot{l}(x,\sigma^{2})=\frac{1}{2\sigma^{2}}-\frac{x^{2}}{2\sigma^{4}}.$
In \eqref{equ9}, we find that for  ${\sigma^{2},{{\sigma}^{*2}} \in\Theta}$
$|\dot{l}(x,\sigma^{2})-\dot{l}(x,{{\sigma}^{*}}^{2})| = |\frac{(\sigma^{2}+{{\sigma}^{*}}^{2})x^{2}-\sigma^{2}{{\sigma}^{*}}^{2}}
{2\sigma^{4}{{\sigma}^{*}}^{4}}(\sigma^{2}-{{\sigma}^{*}}^{2})|\leq c(x)
|\sigma^{2}-{{\sigma}^{*}}^{2}|,$ where $c(x)=\sup_{\sigma^{2},{{\sigma}^{*2}} \in\Theta}|\frac{(\sigma^{2}+{{\sigma}^{*}}^{2})x^{2}-\sigma^{2}{{\sigma}^{*}}^{2}}
{2\sigma^{4}{{\sigma}^{*}}^{4}}|$. We need ${\rm{E}}c^{2}(X_{1})\propto {\rm{E}}X_{1}^2<\infty$.

(b) If $X_1$ is Pareto r.v., whose probability density is $p(x,k)=\frac{kx_{\text{min}}^{k}}{x^{k+1}}\cdot {\rm{1}}_{x>x_{\text{min}}}$, where $x_{\text{min}}$ is the smallest possible value of $x$ and $x_{\text{min}},~k>0$. Thus, for ${k,k^*\in\Theta}$, it gives
$l(x,k)=-\log k-k\log x_{\min}+(k+1)\log x$, $\dot{l}(x,k)=-\frac{1}{k}+\log\frac{x}{x_{\min}},
~\dot{l}(x,k^*)=-\frac{1}{k^*}+\log\frac{x}{x_{\min}}.$
From \eqref{equ9}, we have
$|\dot{l}(x,k)-\dot{l}(x,k^*)|=|(kk^*)^{-1}(k-k^*)|\leq c(x)|k-k^*|$, with $c(x)=\sup\nolimits_{k,k^*\in\Theta}(kk^*)^{-1}$ being constant.

(c) If $X_1$ is Weibull r.v. with probability density $p(x,\lambda,k)=\frac{k}{\lambda}(\frac{x}{\lambda})^{k-1}{e^{-(x/\lambda)}}^{k}\cdot {\rm{1}}_{x\geq0}$, where $\lambda > 0$ is the scale parameter and $k > 0$ is the shape parameter. Weibull distribution is an exponential distribution when $k=1$ and a Rayleigh distribution when $k=2$. For simplicity, assuming that $k\geq2$ is a known parameter, we get
$l(x,\lambda)=-\log k +\log\lambda+(k-1)\log x-(k-1)\log \lambda-x^{k}\lambda^{-k}$ and
$\dot{l}(x,\lambda)=\frac{2-k}{\lambda}+\frac{k x^{k}}{\lambda^{k+1}}.$ By \eqref{equ9},
$|\dot{l}(x,\lambda)-\dot{l}(x,\lambda^{*})|=|(\frac{(k-2)\lambda^{k}{\lambda^{*}}^{k}-kx^{k}
\sum_{i=0}^{k}{\lambda^{*}}^{k-i}\lambda^{i}}{\lambda^{k+1}{\lambda^*}^{k+1}})(\lambda-\lambda^*)|
\leq c(x)|\lambda-\lambda^*|$ for ${\lambda,\lambda^*\in\Theta}$, where $c(x)=\sup_{\lambda,\lambda^*\in\Theta}|\frac{(k-2)\lambda^{k}{\lambda^{*}}^{k}-kx^{k}
\sum_{i=0}^{k}{\lambda^{*}}^{k-i}\lambda^{i}}{\lambda^{k+1}{\lambda^*}^{k+1}}|$. We need ${\rm{E}}c^{2}(X_{1})\propto {\rm{E}}X_{1}^k<\infty$.

\begin{corollary}\label{cX}
For any one-dimensional exponential family distribution with density
\begin{align}\label{equ10}
p(x,\theta)=h(x)\exp\{\theta T(x)-A(\theta)\},
\end{align}
where $h(x)>0$, $T(x)$ and $A(\theta)$ [$\theta$ is called the natural parameter of the exponential family] are known functions. Then the choice of $c(x)$ in Theorem \ref{eq:cantoniG} is a constant independent of $x$.
\end{corollary}

\section{Proofs}
\begin{proof}[Proof of Corollary \ref{lem:maurer2021some}]
Let the \emph{tilted expectation} $\mathrm{E}_{Y}$ be ${\mathrm{E}_Y}[Z] = \mathrm{E}\left[ {Z \cdot \frac{{{e^Y}}}{{\mathrm{E}{e^Y}}}} \right]$ with the exponential weighted v.s. ${\frac{{{e^Y}}}{{\mathrm{E}{e^Y}}}}$. The proof is based on the entropy  of a r.v. $Y$ defined by
$S(Y):=\mathrm{E}_{Y}[Y]-\log \mathrm{E}\left[e^{Y}\right]=\mathrm{E}\left[ {Y \cdot \frac{{{e^Y}}}{{\mathrm{E}{e^Y}}}} \right]-\log \mathrm{E}\left[e^{Y}\right],$
which is free of centering, i.e. $S(Y-\mathrm{E}Y)=S(Y)$. Suppose that $Y$ has zero mean, we have by Jensen's inequality
\begin{align}\label{eq:entropyin}
S(Y)=\mathrm{E}_{Y}\left[\log(\frac{e^{Y}}{\mathrm{E}e^{Y}})\right]  \le \log \mathrm{E}_{Y}\left[\frac{e^{Y}}{\mathrm{E}e^{Y}}\right]=\log \mathrm{E}\left[\frac{e^{Y}{{e^Y}}}{\mathrm{E}e^{Y}{\mathrm{E}{e^Y}}}\right]\le \log \mathrm{E}e^{2 Y},
\end{align}
where the last inequality is also derived by  Jensen's inequality ${\rm{E}}{e^Y} \ge {e^{{\rm{E}}Y}} = 1.$

From (5) in \cite{maurer2021some}, the concentration inequality is by Cramer-Chernoff method with the logarithm of the
MGF represented as the integral of entropy
\begin{align}\label{eq:logmgf}
\log \mathrm{E}\left[e^{t(Y-\mathrm{E}[Y])}\right] = \log \mathrm{E}\left[e^{t Y}\right] = t \int_{0}^{t} {S(\gamma Y) }/{\gamma^{2}}d \gamma,~\forall~t >0
\end{align}
and the \emph{subadditivity of entropy} (see (6) in \cite{maurer2021some})
\begin{align}\label{eq:subadditivity}
S(f(Z)) \le {{\rm{E}}_{f(Z)}}[ {\sum\nolimits_{k = 1}^n S ({D_{f,{Z_k}}}(Z))}],
\end{align}
where we denote $S({D_{f,{Z_k}}}(Z)):=S({D_{f,{Z_k}}}(z))|_{z=Z}$. Next, the proof has two steps.

\textbf{Step 1}. Show that $S(Y) \le \frac{{\left\| Y \right\|_{{\theta_1}}^2}}{{{{(1 - \left\| Y \right\|_{{\theta_1}}^{})}^2}}}$ if ${\left\| Y \right\|_{{\theta_1}}}<1$ and $\mathrm{E} Y=0$.

 By the definition of the tilted expectation, we have
\begin{align}\label{eq:tilt}
&~~~~{{\rm{E}}_{sY}}\left[ {{{\left( {Y - {{\rm{E}}_{sY}}[Y]} \right)}^2}} \right]=\frac{{{\rm{E}}\left[ {\{ {Y^2} - {\rm{2}}Y\left( {{{\rm{E}}_{sY}}[Y]} \right) + {{\left( {{{\rm{E}}_{sY}}[Y]} \right)}^2}\} {e^{sY}}} \right]}}{{{\rm{E}}{e^{sY}}}}\nonumber\\
& = \frac{1}{{{\rm{E}}{e^{sY}}}}\left( {{\rm{E[}}{Y^2}{e^{sY}}] - {\rm{2E[}}Y{e^{sY}}] \cdot \frac{{{\rm{E[}}Y{e^{sY}}]}}{{{\rm{E}}{e^{sY}}}} + {{\left( {\frac{{{\rm{E[}}Y{e^{sY}}]}}{{{\rm{E}}{e^{sY}}}}} \right)}^2\cdot \rm{E}{e^{sY}}}} \right)\nonumber\\
& = \frac{{{\rm{E[}}{Y^2}{e^{sY}}]}}{{{\rm{E}}{e^{sY}}}} - {\rm{2}}\frac{{{\rm{E[}}Y{e^{sY}}]}}{{{\rm{E}}{e^{sY}}}} \cdot \frac{{{\rm{E[}}Y{e^{sY}}]}}{{{\rm{E}}{e^{sY}}}} + {\left( {\frac{{{\rm{E[}}Y{e^{sY}}]}}{{{\rm{E}}{e^{sY}}}}} \right)^2} \nonumber\\
&= \frac{{{\rm{E[}}{Y^2}{e^{sY}}]}}{{{\rm{E}}{e^{sY}}}} - {\left( {\frac{{{\rm{E[}}Y{e^{sY}}]}}{{{\rm{E}}{e^{sY}}}}} \right)^2}\le {{\rm{E}}_{sY}}\left[ {{Y^2}} \right]= \frac{{{\rm{E}}\left[ {{Y^2}{e^{sY}}} \right]}}{{{\rm{E}}{e^{sY}}}}\le {\rm{E}}\left[ {{Y^2}{e^{sY}}} \right],
\end{align}
where the last inequality is from Jensen's inequality $\mathrm{E} e^{\theta Y} \geq e^{\theta \mathrm{EY}}=1$ due to $\mathrm{E} Y=0$.

By ${\left\| Y \right\|_{{\theta_1}}} = \sup_{k \ge 1} {\left( {{{{\rm{E}}|Y|^k}}/{{k!}}} \right)^{1/k}}$, we have
\begin{align}\label{eq:sbEMGF}
{\rm{E}}\left[ {{Y^2}{e^{sY}}} \right] &= {\rm{E}}\left[ {\sum\limits_{{\rm{k}} = 0}^\infty  {\frac{{{s^k}}}{{k!}}} {Y^{k + 2}}} \right] \leq \sum\limits_{k = 0}^\infty  {\frac{{{s^k}{\rm{E|}}Y{|^{k + 2}}}}{{k!}}} \le \mathop {\sup }\limits_{k} \sum\limits_{k = 0}^\infty  {\frac{{{s^k}{\rm{E|}}Y{|^{k + 2}}}}{{k!}}}\\
&= \sum\limits_{k = 0}^\infty  {\frac{{{s^k}(k + 2)!\left\| Y \right\|_{{\theta_1}}^{k + 2}}}{{k!}}}=\left\| Y \right\|_{{\theta_1}}^2\sum\limits_{k = 0}^\infty  {(k + 2)(k + 1)\left\| Y \right\|_{{\theta_1}}^k{s^k}}.
\end{align}

Note that $\int_{0}^{1} \int_{t}^{1} s^{k} d s d t=\frac{1}{k+2}$, the fluctuation representation of entropy implies
\begin{align}\label{eq:sy}
S(Y)& = \int_0^1 {\left( {\int_t^1 {{{\rm{E}}_{sY}}} \left[ {{{\left( {Y - {{\rm{E}}_{sY}}[Y]} \right)}^2}} \right]ds} \right)} dt\nonumber\\
[\text{By}~\eqref{eq:tilt}]~& \le \int_0^1 {\left( {\int_t^1 {\rm{E}} \left[ {{Y^2}{e^{sY}}} \right]ds} \right)} dt \leq \left\| Y \right\|_{{\theta_1}}^2\sum\limits_{k = 0}^\infty  {(k + 2)(k + 1)\left\| Y \right\|_{{\theta_1}}^k\int_0^1 {\int_t^1 {{s^k}} ds} dt} \nonumber\\
& = \left\| Y \right\|_{{\theta_1}}^2\sum\limits_{k = 0}^\infty  {(k + 1)\left\| Y \right\|_{{\theta_1}}^k}  = \frac{{\left\| Y \right\|_{{\theta_1}}^2}}{{{{(1 - \left\| Y \right\|_{{\theta_1}}^{})}^2}}},0 < \left\| Y \right\|_{{\theta_1}} < 1.
\end{align}

\textbf{Step 2}. By the subadditivity of entropy \eqref{eq:subadditivity} and the integral of entropy \eqref{eq:logmgf}, we have
\begin{align*}
&~~~~\log {\rm{E}}\left[ {{e^{t(f(Z) - {\rm{E}}f(Z))}}} \right] = t\int_0^t {\frac{{S(\gamma f(Z))d\gamma }}{{{\gamma ^2}}}}  \le t\int_0^t {\frac{1}{{{\gamma ^2}}}} {{\rm{E}}_{\gamma f(Z)}}\left[ {\sum\limits_{i = 1}^n S \left( {\gamma {D_{f,{Z_i}}}(Z)} \right)} \right]d\gamma \\
& \le t\int_0^t {\frac{{{\gamma ^2}}}{{{\gamma ^2}}}} {{\rm{E}}_{\gamma f(Z)}}\left[ {\sum\limits_{i = 1}^n {\frac{{\left\| {{D_{f,{Z_i}}}(Z)} \right\|_{{\theta_1}}^2}}{{{{(1 - \gamma \left\| {{D_{f,{Z_i}}}(Z)} \right\|_{{\theta_1}}^{})}^2}}}} } \right]d\gamma ,~\forall~\left\| {{D_{f,{Z_i}}}(Z)} \right\|_{{\theta_1}} < 1\\
& \le t\int_0^t  {{\rm{E}}_{\gamma f(Z)}}\left[ {\sum\limits_{i = 1}^n {\frac{{\left\| {{D_{f,{Z_i}}}(Z)} \right\|_{{\theta_1}}^2}}{{{{(1 - \gamma \mathop {\max }\limits_{i \in [n]} \mathop {\sup }\limits_{z \in \mathcal{Z}} \left\| {{D_{f,{Z_i}}}(z)} \right\|_{{\theta_1}})}^2}}}} } \right]d\gamma\\
& \le \int_0^t\frac{{t {{{\rm{E}}_{\gamma f(Z)}}[\mathop {\sup }\limits_{z \in \mathcal{Z}} \sum\limits_{i = 1}^n {\left\| {{D_{f,{Z_i}}}(z)} \right\|_{{\theta_1}}^2} ]}  }}{{{{(1 - \gamma \mathop {\max }\limits_{i \in [n]} \mathop {\sup }\limits_{z \in \mathcal{Z}} \left\| {{D_{f,{Z_i}}}(z)} \right\|_{{\theta_1}}^{})}^2}}}d\gamma = \frac{{2\mathop {\sup }\limits_{z \in \mathcal{Z}} \sum\limits_{i = 1}^n {\left\| {{D_{f,{Z_i}}}(z)} \right\|_{{\theta_1}}^2} }}{{{1 - t \mathop {\max }\limits_{i \in [n]} \mathop {\sup }\limits_{z \in \mathcal{Z}} \left\| {{D_{f,{Z_i}}}(z)} \right\|_{{\theta_1}}}}} \cdot \frac{{{t^2}}}{2},
\end{align*}
where $0< \gamma  \le t < ( {\mathop {\max }\limits_{i \in [n]} \mathop {\sup }\limits_{z \in Z} \left\| {{D_{f,{Z_i}}}(z)} \right\|_{{\theta_1}}} )^{ - 1}$. The sub-Gamma definition gives
$f(Z) - {\rm{E}}f(Z) \sim \operatorname{sub\Gamma}(2{\mathop {\sup }\limits_{z \in \mathcal{Z}} \sum\limits_{i = 1}^n {\left\| {{D_{f,{Z_i}}}(z)} \right\|_{{\theta_1}}^2} }, {\mathop {\max }\limits_{i \in [n]} \mathop {\sup }\limits_{z \in \mathcal{Z}} \left\| {{D_{f,{Z_i}}}(z)} \right\|_{{\theta_1}}})$. Finally, the \eqref{eq:SUBG} and \eqref{eq:SUBG1} obtained have similar conclusions to Lemma 5.1 in \cite{zhang2020concentration}.
\end{proof}

\begin{corollary}\label{cor:normsgse}
Let $\{X_{i}\}_{i=1}^n$ be independent $\mathcal{X}$-valued r.v.s. distributed as $\mu_{i}$ in a metric probability space $(\mathcal{X}, d)$. Put $X=\left(X_{1}, \ldots, X_{n}\right)$ and $X^{\prime}$ is an i.i.d. copy of $X$. Suppose $f: \mathcal{X}^{n} \rightarrow \mathbb{R}$ has \emph{conditional Lipschitz constant} $L$ w.r.t. the metric $\rho$ on $\mathcal{X}^{n}$ defined by $\rho(x, y)=\sum_{i=1}^n d\left(x_{i}, y_{i}\right)$, i.e.
\begin{equation}\label{eq:CSLC}
    \|\mathrm{E}[f\left(S_{X_{k}}^{k} (x)\right)-f(S_{X_{k}^{\prime}}^{k} (x))|X_k] \|_{\vartheta} \le L\|\mathrm{E}[d(X_k, X_k^{\prime})|X_k]\|_\vartheta,~\vartheta=\theta_{1}~\text{or}~\theta_{2},
\end{equation}
where $S_{X_{k}}^{k} (x)$ is the substitution operator defined in \eqref{eq:identy}. Then we have
\begin{center}
$f(X)-\mathrm{E}f(X) \sim \operatorname{subG}(8{ L^{2} \sum\limits_{i=1}^{n} \|d(X_i, X_i^{\prime})\|_{\theta_{2}}^2})$ if ${\mathop {\max }\limits_{i \in [n]} \left\| d(X_i, X_i^{\prime}) \right\|_{\theta_{2}}} <\infty$,
\end{center}
 and
$f(X)-\mathrm{E}f(X) \sim \operatorname{sub\Gamma}(2L^2{\sum_{i = 1}^n {\left\| d(X_i, X_i^{\prime}) \right\|_{\theta_{1}}^2} }, L{\mathop {\max }\limits_{i \in [n]} \left\| d(X_i, X_i^{\prime}) \right\|_{\theta_{1}}})$ if ${\mathop {\max }\limits_{i \in [n]} \left\| d(X_i, X_i^{\prime}) \right\|_{\theta_{1}}} < \infty$.
\end{corollary}
\begin{proof}[Proof of Corollary \ref{cor:normsgse}]
From the identity in \eqref{eq:identy}, we have
\begin{align}\label{eq:identy2}
\|D_{f,  X_k}(z) \|_{\vartheta} &=  \|f\left(x_{1}, \ldots, x_{k-1}, X_{k}, x_{k+1}, \ldots, x_{n}\right) - \mathrm{E}\left[f\left(x_{1}, \ldots, x_{k-1}, X_{k}^{\prime}, x_{k+1}, \ldots, x_{n}\right)\right] \|_{\vartheta}\nonumber\\
&=\|\mathrm{E}[f\left(S_{X_{k}}^{k} (x)\right)-f(S_{X_{k}^{\prime}}^{k} (x))|X_k] \|_{\vartheta}\le L\|\mathrm{E}[d(X_k, X_k^{\prime})|X_k]\|_{\vartheta}~[\text{By~}\eqref{eq:CSLC}].
\end{align}
The conditional Jensen's inequality shows
\begin{align}\nonumber
\mathrm{E}\left[\left|\mathrm{E}\left[d({X_k},{Y_k})\mid {X_k}\right]\right|^{p}\right]&\le \mathrm{E}\left[\mathrm{E}\{\left|d({X_k},{Y_k})\right| \mid {X_k}\}^{p}\right] = \mathrm{E}\left[\mathrm{E}\{\left(\left|d({X_k},{Y_k})\right|^{p}\right)^{1/p} \mid {X_k}\}^{p}\right]\label{eq6}\\
& \le \mathrm{E}\left\{\mathrm{E}\left[\left|d({X_k},{Y_k})\right|^{p} \mid {X_k}\right]\right\}
= \mathrm{E}\left|d({X_k},{Y_k})\right|^{p},~p \geq 1.
\end{align}
Then the definitions of
${\left\| X \right\|_{\theta_{2}}} = \sup\nolimits_{k \ge 1} {[ {\frac{{{2^k}k!}}{{(2k)!}}{\rm{E}}{X^{2k}}} ]^{1/(2k)}}$ and ${\left\| X \right\|_{\theta_{1}}} =\sup\nolimits_{k \ge 1} {( {\frac{{{\rm{E}}|X{|^k}}}{{k! }}} )^{1/k}}$ imply $\|D_{f,  X_k}(z) \|_{\vartheta}  \le L \|d(X_k, X_k^{\prime})]\|_{\vartheta}$ by \eqref{eq:identy2} and \eqref{eq6}. The result follows by applying Proposition \ref{cor:fxcon} and Corollary \ref{lem:maurer2021some}.
\end{proof}

\begin{proof}[Proof of Theorem \ref{thm:MLE}]
Let $X=\left(X_{1}, \ldots, X_{n}\right)$ and denote ${\ddot l_n}(X,\theta ){\rm{ = }}\frac{{{\partial ^2}{l_n}(X,\theta )}}{{\partial {\theta ^2}}}$ with $l_n (X, \theta) := \frac{1}{n} \sum_{i = 1}^n l(X_i, \theta)$. We will show $\widehat \theta=f(X)$ satisfies \eqref{eq:CSLC} with  $L=\frac{{c_l}}{{{c_H}n}}$. Consider $\tilde{\theta}$ as the MLE under turbulence, i.e.
$\dot{l}_n(S_{Y_k}^k (X), \tilde{\theta}) = 0$. By Hessian matrix assumptions and ${{\dot l}_n}(X,\widehat \theta )=0$, there exists $\mathord{\buildrel{\lower3pt\hbox{$\scriptscriptstyle\frown$}} \over \theta }$ between $\widehat{\theta}$ and $\tilde{\theta}$,
\begin{center}
$ \big| {{\dot l}_n}(X,\tilde \theta ) \big| = \big| {{\dot l}_n}(X,\widehat \theta ) + {{\ddot l}_n}(X,\mathord{\buildrel{\lower3pt\hbox{$\scriptscriptstyle\frown$}} \over \theta } )({\tilde \theta  - \widehat \theta }) \big| = \big| {{\ddot l}_n}(X,\mathord{\buildrel{\lower3pt\hbox{$\scriptscriptstyle\frown$}} \over \theta } )({\tilde \theta  - \widehat \theta }) \big| \ge {c_H} |{\tilde \theta  - \widehat \theta }|$
\end{center}
and Lipschitz condition for the score function w.r.t. data
\begin{align*}
    \big| {{\dot l}_n}(X,\tilde \theta ) \big| &= \big| {{\dot l}_n}(X,\tilde \theta ) + {{\dot l}_n}(S_{{Y_k}}^k(X),\tilde \theta ) - {{\dot l}_n}(S_{{Y_k}}^k(X),\tilde \theta ) \big| \\
    &= \big| {{\dot l}_n}(X,\tilde \theta ) - {{\dot l}_n}(S_{{Y_k}}^k(X),\tilde \theta ) \big| \le {c_l} \rho (X,S_{{Y_k}}^k(X))/n = {c_l}d({X_k},{Y_k})/n.
\end{align*}
Thus
$\big| f(S_{{Y_k}}^k(X)) - f(S_{{X_k}}^k(X)) \big| = \big| f(S_{{Y_k}}^k(X)) - f(X) \big| = | \tilde \theta-\widehat \theta | \le \frac{{c_l}{d({X_k},{Y_k})}}{{{c_H}n}}.$\\

Then the \eqref{eq:identy} implies
$\mathrm{E}[| f(S_{X_{k}}^{k} (x))-f(S_{Y_k}^{k} (x)) | \, | \, X_k ]\le \frac{{c_l}\mathrm{E}[d({X_k},{Y_k})|X_k]}{{{c_H}n}}$.
By Corollary \ref{cor:normsgse} with \eqref{eq:CSLC} and $L=\frac{{c_l}}{{{c_H}n}}$, if $\{d({X_k},{Y_k})\}_{k=1}^n$ have finite ${\left\| \cdot \right\|_{\theta_{2}}}$-norm, we get $\widehat \theta-\mathrm{E}\widehat \theta \sim \operatorname{subG}(16c_{l}^{2}/({c_{H}^{2}n^{2}}) \sum_{k = 1}^n {\left\| d({X_k},{Y_k}) \right\|_{{\theta _2}}^2} )$ and \eqref{eq:subG}. If $\{d({X_k},{Y_k})\}_{k=1}^n$ have finite ${\left\| \cdot \right\|_{\theta_{1}}}$-norm, we get $\widehat \theta-\mathrm{E}\widehat \theta\sim \operatorname{sub\Gamma}({2 c_l^2}/{(c_H^2n^2)} \sum_{k = 1}^n {\left\| d({X_k},{Y_k}) \right\|_{{\theta _1}}^2},\\ {c_l}/{(c_Hn)}\max_{1\le k \le n}{\left\| d({X_k},{Y_k}) \right\|_{{\theta _1}}})$ and  \eqref{eq:subE}. So we obtain the result.
\end{proof}

\begin{proof}[Proof of Corollary \ref{cor:error}]
From \eqref{eq:subG} in Theorem \ref{thm:MLE}, it is easy to verify that there is $\delta\in(0,1)$ such that $|\widehat \theta-\mathrm{E}\widehat \theta|\leq\frac{4c_{l}}{nc_{H}}\sqrt{\sum_{k=1}^{n}\|d(X_{k},Y_{k})\|_{\theta_{2}}^{2}\log(2/\delta)}$ holds with high probability. Using (8) in \cite{mardia1999bias}, we have an expression for the bias of MLE
\begin{center}
${\rm{E}}\widehat{\theta}-{\theta ^*}=\frac{1}{2nI^{2}(\theta^*)}\left[2{\rm{E}}(\dot{l}(X,\theta^*)\ddot{l}(X,\theta^*))-{\rm{E}}\dddot{l}(X,\theta^*)\right]
+o(\frac{1}{n})=\frac{\kappa(X,\theta^*)}{2nI^{2}(\theta^*)}+o(\frac{1}{n}),$
\end{center}
where $\kappa(X,\theta^*)=2{\rm{E}}(\dot{l}(X,\theta^*)\ddot{l}(X,\theta^*))-{\rm{E}}\dddot{l}(X,\theta^*)$. Thus, in the sub-Gaussian sample case, we obtain
\begin{align*}
|\widehat{\theta}-{\theta ^*}|=|(\widehat{\theta}-{\rm{E}}\widehat{\theta})+({\rm{E}}\widehat{\theta}-{\theta ^*})|
&\leq |\widehat \theta-\mathrm{E}\widehat \theta|+|{\rm{E}}\widehat{\theta}-{\theta ^*}|\\
&=\frac{4c_{l}}{\sqrt{n}c_{H}}\sqrt{\frac{1}{n}\sum_{k=1}^{n}\|d(X_{k},Y_{k})\|_{\theta_{2}}^{2}\log(\frac{2}{\delta})}
+\frac{|\kappa(X,\theta^*)|}{2nI^{2}(\theta^*)}+o(\frac{1}{n}).
\end{align*}
Similarly, by \eqref{eq:SUBG}, there exists $\delta\in(0,1)$ such that $|\widehat \theta-\mathrm{E}\widehat \theta|\leq \frac{2c_{l}}{nc_{H}}\sqrt{\sum_{i=1}^{n}\|d(X_{k},Y_{k})\|_{\theta_{1}}^{2}\log(\frac{1}{\delta})}
+\frac{c_{l}\max_{1\leq k \leq n}\|d(X_{k},Y_{k})\|_{\theta_{1}}\log(\frac{1}{\delta})}{nc_{H}}$ holds with high probability.
For the sub-exponential case, we have \eqref{eq:subEbias}.
\end{proof}

\begin{proof}[Proof of Theorem \ref{eq:cantoniG}]
Note that $\widehat{\theta}_{\beta}$ is defined as the solution of ${{\widehat Z}_\beta }(\theta )=0$ with the choice of $\psi(x)$ in \eqref{eq1}. Since $\{X_{i}\}_{i=1}^{n}$ is i.i.d and that $\psi(x)\leq\log(1+x+x^{2}/2)$ for all $x\in\mathbb{R}$, then for all $\theta\in \Theta$
\begin{align}\nonumber
{\rm{E}}[e^{{n\beta}{{\widehat Z}_\beta }(\theta )}]&={\rm{E}}[e^{\sum_{i=1}^{n}\psi(\beta\dot{l}(X_{i},\theta))}]
\leq \prod_{i=1}^{n}\left({\rm{E}}[1+\beta\dot{l}(X_{i},\theta)+\frac{\beta^{2}\dot{l}^{2}(X_{i},\theta)}{2}]\right)\\\label{equ2}
&=\left(1+\beta {\rm{E}}\dot{l}(X_{1},\theta)+\frac{\beta^{2}}{2}{\rm{E}}\dot{l}^{2}(X_{1},\theta)\right)^{n}\leq\exp\left\{n\beta {\rm{E}}\dot{l}(X_{1},\theta)+\frac{n\beta^{2}}{2}{\rm{E}}\dot{l}^{2}(X_{1},\theta)\right\}.
\end{align}
By ${\rm{E}}\dot{l}(X_{1},\theta^*)=0$, using the Lipschitz condition, we find
\begin{align}\label{equ3}
{\rm{E}}\dot{l}(X_{1},\theta)={\rm{E}}[\dot{l}(X_{1},\theta)-\dot{l}(X_{1},\theta^*)+\dot{l}(X_{1},\theta^*)]\leq {\rm{E}}[c(X_{1})]|\theta-\theta^*|.
\end{align}
Since ${\rm{E}}\dot{l}^{2}(X_{1},\theta^*)={\rm{Var}}(\dot{l}(X_{1},\theta^*))=:I(\theta^*)$, then
\begin{align*}
{\rm{E}}\dot{l}^{2}(X_{1},\theta)&={\rm{E}}[\dot{l}(X_{1},\theta)-\dot{l}(X_{1},\theta^*)+\dot{l}(X_{1},\theta^*)]^{2}\\
&={\rm{E}}[\dot{l}(X_{1},\theta)-\dot{l}(X_{1},\theta^*)]^{2}+2{\rm{E}}[\dot{l}(X_{1},\theta)-\dot{l}(X_{1},\theta^*)]
\dot{l}(X_{1},\theta^*)+{\rm{E}}\dot{l}^{2}(X_{1},\theta^*).
\end{align*}
The following will be divided into three parts for discussion.\\
(i) If $c(X_{1})$ is related to $X_{1}$, based on the basic inequality and Lipschitz condition, we can obtain ${\rm{E}}\dot{l}^{2}(X_{1},\theta)\leq2[(\theta-\theta^*)^{2}{\rm{E}}c^{2}(X_{1})+I(\theta^*)]$. Combining \eqref{equ2} and \eqref{equ3}, then
\begin{align*}
{\rm{E}}[e^{{n\beta}{{\widehat Z}_\beta }(\theta )}]\leq\exp\{n\beta|\theta-\theta^*|{\rm{E}}c(X_{1})+n\beta^{2}(\theta-\theta^*)^{2}{\rm{E}}c^{2}
(X_{1})+n\beta^{2}I(\theta^*)\}.
\end{align*}
Using the Markov’s inequality, we obtain that, for any fixed $\theta\in \Theta$ and $\delta\in(0,1/2)$,
\begin{align}\label{eq2}
    &\pr \left\{{{n\beta}{{\widehat Z}_\beta }(\theta )}\geq n\beta|\theta-\theta^*|{\rm{E}}c(X_{1})+n\beta^{2}(\theta-\theta^*)^{2}{\rm{E}}c^{2}(X_{1})+n\beta^{2}I(\theta^*)
    +\log(\delta^{-1})\right\}\\\nonumber
    &=\pr \{e^{{n\beta}{{\widehat Z}_\beta }(\theta )}\geq e^{n\beta|\theta-\theta^*|{\rm{E}}c(X_{1})+n\beta^{2}(\theta-\theta^*)^{2}{\rm{E}}c^{2}(X_{1})+n\beta^{2}I(\theta^*)
+\log(\delta^{-1})}\}\leq\delta.
\end{align}
For the structure of \eqref{eq2}, let $t=|\theta-\theta^*|$, and then the quadratic function of $t$
\begin{align}\label{equ8}
g(t):=n\beta^{2}{\rm{E}}c^{2}(X_{1})t^{2}+n\beta {\rm{E}}c(X_{1})t+n\beta^{2}I(\theta^*)+\log(\delta^{-1})
\end{align}
has at least one root. In particular, when $n\geq\frac{4{\rm{E}}c^{2}(X_{1})\log(\delta^{-1})}{[{\rm{E}}c(X_{1})]^{2}-4\beta^{2}{\rm{E}}c^{2}(X_{1})I(\theta^*)}>0$, opting the smaller root
\begin{align}\label{equ4}
\theta_{+}=\theta^{*}+\frac{\beta I(\theta^*)+{\log(\delta^{-1})}/{(n\beta)}}{\frac{1}{2}{\rm{E}}c(X_{1})\left(1+\sqrt{1-\frac{4\beta^{2}
{\rm{E}}c^{2}(X_{1})I(\theta^*)}{[{\rm{E}}c(X_{1})]^{2}}-\frac{4{\rm{E}}c^{2}(X_{1})\log(\delta^{-1})}{n[{\rm{E}}
c(X_{1})]^{2}}}\right)},
\end{align}
we obtain that $\pr ({{\widehat Z}_\beta }(\theta_{+})\leq0)\geq1-\delta$. Since ${{\widehat Z}_\beta }(\theta)$ is strictly decreasing w.r.t. $\theta$, this shows that $\pr (\widehat{\theta}_{\beta}\leq\theta_{+})\geq1-\delta$. Similarly, it is easy to prove that $P(\widehat{\theta}_{\beta}\geq\theta_{-})\geq1-\delta$, where
\begin{align}\label{equ5}
\theta_{-}=\theta^*-\frac{\beta I(\theta^*)+{\log(\delta^{-1})}/{(n\beta)}}{\frac{1}{2}{\rm{E}}c(X_{1})\left(1+\sqrt{1-\frac{4\beta^{2} {\rm{E}}c^{2}(X_{1})I(\theta^*)}{[{\rm{E}}c(X_{1})]^{2}}-\frac{4{\rm{E}}c^{2}(X_{1})
\log(\delta^{-1})}{n[{\rm{E}}c(X_{1})]^{2}}}\right)}.
\end{align}
Obviously we have $\pr (\theta_{-} \leq {\widehat \theta }_{\beta } \leq \theta_{+}) \ge 1 - 2\delta$. Let $h(\beta)=\frac{\beta I(\theta^*)+{\log(\delta^{-1})}/{(n\beta)}}{\frac{1}{2}{\rm{E}}c(X_{1})\left(1+\sqrt{1-\frac{4\beta^{2} {\rm{E}}c^{2}(X_{1})I(\theta^*)}{[{\rm{E}}c(X_{1})]^{2}}-\frac{4{\rm{E}}c^{2}(X_{1})
\log(\delta^{-1})}{n[{\rm{E}}c(X_{1})]^{2}}}\right)}$.
(ii) Next, we consider selecting the optimal parameter $\beta$ to make the length of the confidence interval as short as possible, by calculating the minimum value of $h(\beta)$. It is easy to verify that the second derivative of $h(\beta)$ is non-negative (that is, the minimum value of $h(\beta)$ exists). Then the optimal $\beta$ value $\beta_{opt1}=\sqrt{\frac{\log(\delta^{-1})}{nI(\theta^*)(1+\frac{4{\rm{E}}c^{2}(X_{1})
\log(\delta^{-1})}{n[{\rm{E}}c(X_{1})]^{2}-4{\rm{E}}c^{2}(X_{1})\log(\delta^{-1})})}}$ can be obtained from $\dot{h}(\beta_{opt1})=0$.
Substituting $\beta_{opt1}$ into $h(\beta)$, then we have the minimum value $h(\beta)_{\min}:=2\sqrt{\frac{I(\theta^*)\log(\delta^{-1})}{n[{\rm{E}}c(X_{1})]^{2}-4{\rm{E}}c^{2}(X_{1})
\log(\delta^{-1})}}.$
Combining \eqref{equ4} and \eqref{equ5}, we show that \eqref{eq:T1} with probability at least $1-2\delta$.\\
(iii) If $c(x)=c>0$. By the Lipschitz condition and ${\rm{E}}\dot{l}(X_{1},\theta^*)=0$, we can write ${\rm{E}}\dot{l}^{2}(X_{1},\theta)\leq(\theta-\theta^*)^{2}{\rm{E}}c^{2}(X_{1})
+I(\theta^*)+2c(X_{1})|\theta-\theta^*|{\rm{E}}\dot{l}(X_{1},\theta^*)=(\theta-\theta^*)^{2}c^{2}
+I(\theta^*)$. Similar to the discussion in (i) above, pick $n\geq\frac{2{\rm{E}}c^{2}(X_{1})\log(\delta^{-1})}{[{\rm{E}}c(X_{1})]^{2}-\beta^{2}{\rm{E}}c^{2}(X_{1})I(\theta^*)}
=\frac{2\log(\delta^{-1})}{1-\beta^{2}I(\theta^{*})}$, then \eqref{equ4} and \eqref{equ5} can be rewritten as
\begin{align}\label{equ6}\small
\theta_{+}=\theta^*+\frac{\beta I(\theta^*)/2+{\log(\delta^{-1})}/{(n\beta)}}{\frac{1}{2}c\left(1+\sqrt{1-\beta^{2}I(\theta^*)
-\frac{2\log(\delta^{-1})}{n}}\right)},~\theta_{-}=\theta^*-\frac{\beta I(\theta^*)/2+{\log(\delta^{-1})}/{(n\beta)}}{\frac{1}{2}c\left(1+\sqrt{1-\beta^{2}I(\theta^*)
-\frac{2\log(\delta^{-1})}{n}}\right)},
\end{align}
respectively. Similar to the conclusion of (i), we also have $\pr (\theta_{-} \leq {\widehat \theta }_{\beta } \leq \theta_{+}) \ge 1 - 2\delta$. Taking the optimal tuning parameter $\beta_{opt2}=\sqrt{\frac{2\log(\delta^{-1})}{nI(\theta^*)(1+\frac{2{\rm{E}}c^{2}(X_{1})
\log(\delta^{-1})}{n[{\rm{E}}c(X_{1})]^{2}-2{\rm{E}}c^{2}(X_{1})\log(\delta^{-1})})}}=\sqrt{\frac{2
\log(\delta^{-1})}{nI(\theta^*)\left(1+\frac{2\log(\delta^{-1})}{n-2\log(\delta^{-1})}\right)}}$, and using \eqref{equ6}, we obtain \eqref{eq:T1} with probability at least $1-2\delta$.
\end{proof}
\begin{proof}[Proof of Corollary \ref{cX}]
Let $\{X_{i}\}_{i=1}^{n}$ be a sequence of exponential family i.i.d r.v.s whose density satisfies \eqref{equ10}. We obtain $l(x,\theta)=-\log(p(x,\theta))=-\log h(x)-\theta T(x)-A(\theta)$ and $\dot{l}(x,\theta)=-T(x)+\dot{A}(\theta)$. For \eqref{equ9}, the mean value theorem implies
$$|\dot{l}(x,\theta)-\dot{l}(x,\theta^*)|=|\dot{A}(\theta)-\dot{A}(\theta^*)|
=|\ddot{A}(\breve{\theta})(\theta-\theta^*)|\leq c(x)|\theta-\theta^*|,~\theta^*,\theta \in \Theta$$
where $c(x)=\sup_{\breve{\theta}\in \Theta}|\ddot{A}(\breve{\theta})|$ and the value $\breve{\theta}$ is between $\theta$ and $\theta^*$. Clearly, the $c(x)$ we find is a constant independent of $x$, which completes the proof of Corollary \ref{cX}.
\end{proof}
\section*{Acknowledgements}
Xiaowei Yang is supported in part by the Key Project of Natural Science Foundation of Anhui Province Colleges and Universities (KJ2021A1034) and Key Scientific Research Project of Chaohu University (XLZ-202105). The author would like to thank Dr. Huiming Zhang for the discussions of the ideas behind this article.



%
%



%

\bibliographystyle{apalike}
\bibliography{ref}

\end{document}